\documentclass[a4paper,11pt,reqno]{amsart}
\usepackage{amssymb}
\usepackage{amsmath,enumerate}
\usepackage{graphicx}

\usepackage[top=30truemm,bottom=30truemm,left=20truemm,right=20truemm]{geometry}
\allowdisplaybreaks[4]

\theoremstyle{plain}
\newtheorem{thm}{Theorem}[section]
\newtheorem{prop}[thm]{Proposition}
\newtheorem{lem}[thm]{Lemma}
\newtheorem{cor}[thm]{Corollary}

\theoremstyle{definition}

\newtheorem{rmk}[thm]{Remark}
\newtheorem{exa}[thm]{Example}

\newenvironment{Acknowledgement}{%
\bigskip\noindent\textit{Acknowledgement.}}

\makeatletter
\@addtoreset{equation}{section}

\makeatother

\title[Monotonicity of second order difference equations]{
  Monotonicity of solutions to second order linear difference equations with constant coefficients
}
\author[Y. Goto]{Yoshiaki Goto}
\address[Goto]{
Otaru University of Commerce, 
3-5-21, Midori, Otaru, Hokkaido, 047-8501, JAPAN
}
\email{goto@res.otaru-uc.ac.jp}
\author[G. Shibukawa]{Genki Shibukawa}
\address[Shibukawa]{
Kitami Institute of Technology,
165, Koen-cho, Kitami, Hokkaido, 090-8507, JAPAN
}
\email{g-shibukawa@mail.kitami-it.ac.jp}

\date{\today}

\begin{document}
\begin{abstract}
  We describe some monotone properties of solutions to 
  second order linear difference equations with real constant coefficients. 
  As an application, we give a characterization of the Fibonacci numbers.  
\end{abstract}

\maketitle

\section{Introduction}
The Fibonacci numbers
$$
F_{0}:=0, \quad F_{1}:=1, \quad F_{n+2}-F_{n+1}-F_{n}=0
$$
can be regarded as the special values of the complete homogeneous symmetric polynomials 
$$
h_{n}(x,y)
   :=
   \frac{x^{n+1}-y^{n+1}}{x-y}
   =
   \sum_{j=0}^{n}x^{k}y^{n-k}, \quad 
   h_{-1}(x,y):=0, \quad 
   h_{0}(x,y):=1
$$
in two variables with 
$$
x=-2\cos{\left(\frac{4\pi }{5}\right)}=\frac{1+\sqrt{5}}{2}, \quad 
y=-2\cos{\left(\frac{2\pi }{5}\right)}=\frac{1-\sqrt{5}}{2},
$$
that is, 
$$
F_{n+1}
   =
   h_{n}\left(-2\cos{\left(\frac{4\pi }{5}\right)}, -2\cos{\left(\frac{2\pi }{5}\right)}\right)
   =
   \frac{1}{\sqrt{5}}
   \left[\left(\frac{1+\sqrt{5}}{2}\right)^{n+1}-\left(\frac{1-\sqrt{5}}{2}\right)^{n+1}\right] .
$$
For the Fibonacci numbers, we easily check that 
the following inequalities which we call monotone properties hold: 
\begin{itemize}
\item \textbf{monotone property 1} (positive monotone non-decreasing property):
\begin{align}
\label{eq:monotonic 1}
F_{0}=0<  F_{n+1}\leq F_{n+2};
\end{align}
\item \textbf{monotone property 2}:
\begin{align}
\label{eq:monotonic 2}
% \left|\frac{1+\sqrt{5}}{2}-\frac{F_{n+2}}{F_{n+1}}\right|
%    \leq \left|\frac{1+\sqrt{5}}{2}-\frac{F_{n+1}}{F_{n}}\right|.
\left|\frac{1+\sqrt{5}}{2}-\frac{F_{n+1}}{F_{n}}\right|
   \geq \left|\frac{1+\sqrt{5}}{2}-\frac{F_{n+2}}{F_{n+1}}\right| ;
\end{align}
\item \textbf{monotone property 3}:
\begin{align}
\label{eq:monotonic 3}
% \left|F_{n+1}\frac{1+\sqrt{5}}{2}-F_{n+2}\right|
%    \leq \left|F_{n}\frac{1+\sqrt{5}}{2}-F_{n+1}\right|.
\left|F_{n}\frac{1+\sqrt{5}}{2}-F_{n+1}\right|
   \geq \left|F_{n+1}\frac{1+\sqrt{5}}{2}-F_{n+2}\right|.
\end{align}
\end{itemize}

The Lucas numbers
$$
L_{0}:=2, \quad L_{1}:=1, \quad L_{n+2}-L_{n+1}-L_{n}=0
$$
can be regarded as the special values of the power sum symmetric polynomials 
$$
p_{n}(x,y)
   :=
   x^{n}+y^{n}, \quad 
   p_{0}(x,y)=2
$$
in two variables: 
$$
L_{n}
   =
   p_{n}\left(-2\cos{\left(\frac{4\pi }{5}\right)}, -2\cos{\left(\frac{2\pi }{5}\right)}\right)
   =
   \left(\frac{1+\sqrt{5}}{2}\right)^{n}+\left(\frac{1-\sqrt{5}}{2}\right)^{n} .
$$
For the Lucas numbers, since we have 
\begin{align*}
& L_{0}=2 > L_{1}=1 < L_{2}=3 < L_{3}=4 < \cdots , \\
& \left|\frac{1+\sqrt{5}}{2}-\frac{1}{2}\right|=1.11\cdots 
   \leq \left|\frac{1+\sqrt{5}}{2}-3\right|=1.38\cdots 
   \geq \left|\frac{1+\sqrt{5}}{2}-\frac{4}{3}\right|=0.28\cdots 
   \geq \cdots ,
\end{align*}
% (\ref{eq:monotonic 1}) and (\ref{eq:monotonic 2}) do not hold. 
the monotone properties 1 and 2 do not hold. 
However, because of 
\begin{align*}
\left|L_{n}\frac{1+\sqrt{5}}{2}-L_{n+1}\right|
   =
   \left|\frac{1-\sqrt{5}}{2}\right|^{n}
   \left|\frac{1+\sqrt{5}}{2}-\frac{1-\sqrt{5}}{2}\right|
   =
   \sqrt{5}
   \left|\frac{1-\sqrt{5}}{2}\right|^{n}
   =
   \sqrt{5}
   (0.61\cdots )^{n} ,
\end{align*}
the monotone property 3 holds. 

In this paper, we study the monotone properties 1--3 of a solution to 
a second order linear difference equation with real constant coefficients. 
For a first order linear difference equation, the problem is easy. 
Let us consider
% with real constant coefficients, 
% with a non-zero initial value $c_{0}\in \mathbb{R}$: 
\begin{align*}
% \label{eq:one order diff eq}
a_{0}:=c_{0}\in \mathbb{R}\setminus \{0\}, \quad  
a_{n+1}-aa_{n}=0 ,\quad a \in \mathbb{R}\setminus \{0\} .
\end{align*}
For the solution 
$$
a_{n}=c_{0}a^{n},
$$
the monotone property 1
is equivalent to the condition $c_{0}>0$ and $a\geq 1$. 
The monotone properties 2 and 3 always hold, since 
$$
\left|a-\frac{a_{n+1}}{a_{n}}\right|
   =
   0, \quad 
\left|a_{n}a-a_{n+1}\right|=0 .
$$

Now we consider 
a second order linear difference equation with real constant coefficients, 
with initial values $(0,0)\not= (c_{0},c_{1})\in \mathbb{R}^{2}$ or $(0,0)\not=(c_{-1},c_{0}) \in \mathbb{R}^{2}$: 
\begin{align}
\label{eq:two order diff eq}
a_{0}:=c_{0}, \quad a_{1}:=c_{1}, \quad 
a_{n+2}-aa_{n+1}+ba_{n}=0, \quad a,b \in \mathbb{R}, \quad ab\not=0 .
\end{align}
In general, conditions for 
the monotone properties 1--3 of its solution are not so simple. 
% the monotonicities (\ref{eq:monotonic 1})--(\ref{eq:monotonic 3}) of its solution are not so simple. 
As an example, we consider the following two difference equations with initial values $c_{-1}:=0$, $c_{0}:=1$, 
which are given as the special values of the complete homogeneous symmetric polynomials $h_{n}(x,y)$. 
% These can be naturally regarded as generalizations of the Fibonacci numbers. 
% that satisfy (\ref{eq:monotonic 1})--(\ref{eq:monotonic 3}). 

\begin{exa}
\label{exa:count exam1}
% Suppose 
% \begin{align*}
%   % \label{eq:count exam1}
%   a=1, \quad b=\frac{1}{4}, \quad a_{-1}:=0, \quad a_{0}:=1 .
% \end{align*}
The solution to (\ref{eq:two order diff eq}) 
for $a=1$, $b=1/4$, $a_{-1}:=0$, $a_{0}:=1$
is
$$
a_{n}=(n+1)2^{-n} .
$$
Because of 
$$
a_{n}\geq a_{n+1}, 
$$
the monotone property 1 does not hold. 
On the other hand, by 
$$
\left|\frac{1}{2}-\frac{a_{n+1}}{a_{n}}\right|
   =
   \frac{1}{2}\frac{1}{n+1} ,\qquad 
\left|a_{n}\frac{1}{2}-a_{n+1}\right|
   =
   \frac{1}{2^{n+1}} ,
$$
the monotone properties 2 and 3 hold. 
\end{exa}

\begin{exa}
% Suppose 
% \begin{align*}
% % \label{eq:count exam2}
% a=1, \quad b=-3, \quad a_{-1}:=0, \quad a_{0}:=1 ,
% \end{align*}
The solution to (\ref{eq:two order diff eq})
for $a=1$, $b=-3$, $a_{-1}:=0$, $a_{0}:=1$
is
$$
a_{n}
   =
   \frac{1}{\sqrt{13}}\left[\left(\frac{1+\sqrt{13}}{2}\right)^{n+1}-\left(\frac{1-\sqrt{13}}{2}\right)^{n+1}\right].
$$
By the definition (\ref{eq:two order diff eq}), it satisfies the monotone property 1. 
However, the monotone properties 2 and 3 do not hold, since we have
\begin{align*}
&\left|\frac{1+\sqrt{13}}{2}-1\right|=1.30\cdots 
   \leq \left|\frac{1+\sqrt{13}}{2}-4\right|=1.69\cdots 
   \geq \left|\frac{1+\sqrt{13}}{2}-\frac{7}{4}\right|=0.55\cdots
   \geq \cdots ,
\\
&\left|\frac{1+\sqrt{13}}{2}-1\right|=1.30\cdots 
   \leq \left|\frac{1+\sqrt{13}}{2}-4\right|=1.69\cdots 
   \leq \left|4\frac{1+\sqrt{13}}{2}-7\right|=2.21\cdots 
   \leq \cdots .  
\end{align*}
% $$
% \left|\frac{1+\sqrt{13}}{2}-1\right|=1.30\cdots 
%    \leq \left|\frac{1+\sqrt{13}}{2}-4\right|=1.69\cdots 
%    \geq \left|\frac{1+\sqrt{13}}{2}-\frac{7}{4}\right|=0.55\cdots
%    \geq \cdots ,
% $$
% and 
% $$
% \left|\frac{1+\sqrt{13}}{2}-1\right|=1.30\cdots 
%    \leq \left|\frac{1+\sqrt{13}}{2}-4\right|=1.69\cdots 
%    \leq \left|4\frac{1+\sqrt{13}}{2}-7\right|=2.21\cdots 
%    \leq \cdots .
% $$
\end{exa}

Further, as in the case of the Lucas numbers, even if the monotone properties 1 and 2 
do not hold, sometimes the properties hold for sufficiently large $n$. 

Let $a_{n}$ be the solution to the difference equation (\ref{eq:two order diff eq}), and let 
$$
\alpha _{\pm }:=\frac{a\pm \sqrt{a^{2}-4b}}{2}
$$
be the roots of the characteristic polynomial $x^{2}-ax+b$. 
Here, when $a^{2}-4b<0$, we set $\sqrt{a^{2}-4b}=\sqrt{-1}\sqrt{4b-a^{2}}$. 
Note that the assumption $ab\not=0$ implies 
\begin{align}
\label{eq:eq char root nonzero}
&\alpha _{\pm}\not=0,\\
\label{eq:eq cond char root}
&|\alpha _{+}|=|\alpha _{-}| \quad \Leftrightarrow \quad \alpha _{+}=\alpha _{-} .
\end{align}
We also use notations $\alpha$, $\beta$ for the characteristic roots 
(i.e., $\{\alpha ,\beta\}=\{\alpha _{+},\alpha _{-}\}$) 
so that 
$$
|\alpha |\geq |\beta | .
$$

We first give conditions under which 
the following inequalities (\ref{eq:monot1 0}) and (\ref{eq:monot2}) hold for sufficiently large $n$. 
\begin{thm}
\label{thm:monot1-2}
There exists a non-negative integer $n_{0}$ such that 
\begin{align}
\label{eq:monot1 0}
a_{n}\leq a_{n+1} \quad (n_{0}\leq n), 
\end{align}
if and only if $a^{2}-4b\geq 0$ and one of the following two conditions hold: 
\begin{align}
\label{eq:monotonic cond 1}
& 1\not=\alpha _{+}>0,\quad a>0, \quad (\alpha _{+}-1)(c_{1}-c_{0}\alpha _{-})> 0;
\end{align}
or 
\begin{align}
\label{eq:monotonic cond 2}
% & \alpha _{+}=1, \quad a_{0}\leq a_{1}, \quad a\geq 1.
& \alpha _{+}=1, \quad a_{0}\leq a_{1}\leq a_{2}.
\end{align}
\end{thm}

\begin{thm}
\label{thm:monot2-2}
There exists a non-negative integer $n_{0}$ such that 
\begin{align}
\label{eq:monot2}
% \left|\alpha -\frac{a_{n+2}}{a_{n+1}}\right|\leq \left|\alpha -\frac{a_{n+1}}{a_{n}}\right|
\left|\alpha -\frac{a_{n+1}}{a_{n}}\right| \geq \left|\alpha -\frac{a_{n+2}}{a_{n+1}}\right| 
\quad (n_{0}\leq n),
\end{align}
if and only if $a^{2}-4b\geq 0$ holds. 
\end{thm}

We next discuss three monotone properties 1, 2 and 3 (respectively, Theorem \ref{thm:monot1}, Theorem \ref{thm:monot2} and Theorem \ref{thm:monot3}).
\begin{thm}
\label{thm:monot1}
For a non-negative integer $k$, 
the inequality 
\begin{align}
\label{eq:monot1-k}
a_{n}\leq a_{n+1} \quad (n\geq k-1)
\end{align}
holds if and only if $a^{2}-4b\geq 0$ and one of the following two conditions hold: 
\begin{align}
\label{eq:cond1}
a_{k-1}\leq a_{k}\leq a_{k+1}, \quad 1\not=\alpha _{+}>0,\quad a>0, \quad (\alpha _{+}-1)(c_{1}-c_{0}\alpha _{-})> 0,
\end{align}
or 
\begin{align}
\label{eq:cond1 another}
a_{k-1}\leq a_{k}\leq a_{k+1}, \quad \alpha _{+}=1.
\end{align}
In particular, the inequality $a_{n}\leq a_{n+1}$ holds for any non-negative integer $n$ 
if and only if $a^{2}-4b\geq 0$ and one of the following two conditions hold: 
\begin{align}
\label{eq:cond1 monot}
a_{-1}\leq a_{0}\leq a_{1}, \quad 1\not=\alpha _{+}>0,\quad a>0, \quad (\alpha _{+}-1)(c_{1}-c_{0}\alpha _{-})> 0,
\end{align}
or
\begin{align}
\label{eq:cond1 another monot}
a_{-1}\leq a_{0}\leq a_{1}, \quad \alpha _{+}=1.
\end{align}
\end{thm}

As a corollary of Theorem \ref{thm:monot1}, we obtain 
a condition equivalent to 
the monotone property 1 
for the special values of the complete homogeneous symmetric polynomials. 
\begin{cor}
\label{cor:monot1}
Let the initial values be $a_{-1}:=0$, $a_{0}\not=0$. 
The inequality 
\begin{align}
\label{eq:monot1}
a_{-1}=0< a_{0}\leq a_{n}\leq a_{n+1}
\end{align}
holds for any non-negative integer $n$ if and only if
we have $a^{2}-4b \geq 0$ and 
\begin{align}
\label{eq:cond1 cor}
a_{0}>0, \quad 
a\geq 1, \quad 
\alpha _{+}\geq 1.
\end{align}
\end{cor}

\begin{thm}
\label{thm:monot2}
Let the initial values be $a_{-1}:=0$, $a_{0}\not=0$. 
The inequality 
\begin{align}
\label{eq:monot2-2}
    % \left|\alpha -\frac{a_{n+2}}{a_{n+1}}\right|\leq \left|\alpha -\frac{a_{n+1}}{a_{n}}\right|
\left|\alpha -\frac{a_{n+1}}{a_{n}}\right| \geq \left|\alpha -\frac{a_{n+2}}{a_{n+1}}\right|
\end{align}
holds for any non-negative integer $n$ if and only if
we have $a^{2}-4b \geq 0$ and 
\begin{align}
\label{eq:cond2}
|\alpha +\beta |=|a|\geq |\beta |.
% \left|\frac{c_{0}\alpha }{c_{1}}\right|\leq 1.
\end{align}
\end{thm}

\begin{thm}
\label{thm:monot3}
The inequality 
\begin{align}
\label{eq:monot3}
% \left|a_{n+1}\alpha -a_{n+2}\right|\leq \left|a_{n}\alpha -a_{n+1}\right|
\left|a_{n}\alpha -a_{n+1}\right| \geq \left|a_{n+1}\alpha -a_{n+2}\right|
\end{align}
holds for any non-negative integer $n$ if and only if
\begin{align}
\label{eq:cond3}
|\beta |\leq 1.
\end{align}
\end{thm}

This paper is arranged as follows. 
In Section \ref{section:preliminaries}, 
we give some lemmas which are used in the proof of main results. 
In Section \ref{section:proof}, 
we prove our main results. 
In Section \ref{section:concluding},
we determine a class of special values of $h_n(x,y)$ which satisfy
the monotone properties 1--3, 
and give a characterization of the Fibonacci numbers.

\section{Preliminaries}\label{section:preliminaries}
Throughout this paper, $a_{n}$ is the solution to the difference equation (\ref{eq:two order diff eq}). 
\begin{lem}
\label{thm:charact roots ineq}
We assume $a^{2}-4b\geq 0$. We have  
\begin{align*}
|\alpha _{+}|-|\alpha _{-}|
   &=
   \begin{cases}
   \sqrt{a^{2}-4b} & (a>0, b>0) \\
   -\sqrt{a^{2}-4b} & (a<0, b>0) \\
   a & (b<0)
   \end{cases}.
\end{align*}
Especially, we obtain 
\begin{align}
\label{eq:alpha ineq}
a>0, \, a^{2}-4b>0 \quad &\Leftrightarrow \quad |\alpha _{+}|=\alpha _{+}>|\alpha _{-}|, \\
\label{eq:alpha ineq2}
a<0, \, a^{2}-4b>0 \quad &\Leftrightarrow \quad |\alpha _{+}|<|\alpha _{-}|.
\end{align}
\end{lem}
\begin{proof}
We first assume $a>0$ and $b>0$. Then we have $a>\sqrt{a^{2}-4b}\geq 0$, and hence, we obtain
$$
|\alpha _{+}|-|\alpha _{-}|
   =
   \alpha _{+}-\alpha _{-}
   =
   \sqrt{a^{2}-4b}.
$$
Next, we assume $a<0$ and $b>0$. Because of $a+\sqrt{a^{2}-4b}<0$ and $a-\sqrt{a^{2}-4b}<0$, we obtain 
$$
|\alpha _{+}|-|\alpha _{-}|
   =
   -\alpha _{+}+\alpha _{-}
   =
   -\sqrt{a^{2}-4b}.
$$
Finally, we assume $b<0$. Then we have $\sqrt{a^{2}-4b}>a$, which implies 
$$
|\alpha _{+}|-|\alpha _{-}|
   =
   \alpha _{+}+\alpha _{-}
   =
   a.
$$
\end{proof}

\begin{lem}
For any non-negative integer $n$, we have 
\begin{align}
\label{eq:expression 1}
a_{n}
   &=
   c_{0}h_{n}(\alpha _{+},\alpha _{-})+(c_{1}-ac_{0})h_{n-1}(\alpha _{+},\alpha _{-}) \\
  \nonumber
   &=
   c_{0}h_{n}(\alpha _{+},\alpha _{-})-bc_{-1}h_{n-1}(\alpha _{+},\alpha _{-}) \\
  \nonumber
   &=
   \frac{(c_{0}\alpha _{+}+c_{1}-ac_{0})\alpha _{+}^{n}-(c_{0}\alpha _{-}+c_{1}-ac_{0})\alpha _{-}^{n}}{\sqrt{a^{2}-4b}} \\
\label{eq:expression 1 final}
   &=
   \frac{(c_{1}-c_{0}\alpha _{-})\alpha _{+}^{n}-(c_{1}-c_{0}\alpha _{+})\alpha _{-}^{n}}{\sqrt{a^{2}-4b}}.
\end{align}
In particular, in the case of $\alpha =\alpha _{+}=\alpha _{-}$, we have 
\begin{align}
\label{eq:expression 1-1}
a_{n}
   &=
   c_{0}(n+1)\alpha ^{n}+(c_{1}-ac_{0})n\alpha ^{n-1} \\
\nonumber
   &=
   c_{0}(1-n)\alpha ^{n}+c_{1}n\alpha ^{n-1}.
\end{align}
\end{lem}
\begin{proof}
We prove only (\ref{eq:expression 1}); the others easily follow from this. 
By the definition, we have $h_{-1}(\alpha _{+},\alpha _{-})=0$, 
$h_{0}(\alpha _{+},\alpha _{-})=1$ and $h_{1}(\alpha _{+},\alpha _{-})=a$. 
Thus, (\ref{eq:expression 1}) holds for $n=0,1$. 
Further, we have 
\begin{align*}
& h_{n+2}(\alpha _{+},\alpha _{-})-ah_{n+1}(\alpha _{+},\alpha _{-})+bh_{n}(\alpha _{+},\alpha _{-}) \\
   & \quad =
   \frac{(\alpha _{+}^{n+3}-a\alpha _{+}^{n+2}+b\alpha _{+}^{n+1})-(\alpha _{-}^{n+3}-a\alpha _{-}^{n+2}+b\alpha _{-}^{n+1})}{\alpha _{+}-\alpha _{-}}
   =0 ,
\end{align*}
which implies that (\ref{eq:expression 1}) holds for any non-negative integer $n$. 
\end{proof}

\begin{lem}
\label{thm:Goto lemma}
If we assume $c_{0}c_{1}\not=0$ and $a^{2}-4b\geq 0$, 
then the solution $a_{n}$ of the difference equation (\ref{eq:two order diff eq}) is not zero, 
except for at most one $n$. 
In particular, we have $a_{n}\not=0$ for sufficiently large $n$. 
\end{lem}
\begin{proof}
First, we assume $a^{2}-4b=0$. 
In this case, we have $\alpha _{+}=\alpha _{-}=\alpha $, and hence, 
$$
a_{n}=c_{0}(1-n)\alpha ^{n}+c_{1}n\alpha ^{n-1}=\alpha ^{n-1}(c_{0}\alpha -(c_{0}\alpha -c_{1})n)=0
$$
implies $c_{0}\alpha =(c_{0}\alpha -c_{1})n$. 
We may assume $c_{0}\alpha -c_{1}\not=0$, since $c_{0}\alpha -c_{1}=0$ implies $c_{0}=c_{1}=0$. 
We obtain 
$$
n=\frac{c_{0}\alpha }{c_{0}\alpha -c_{1}} ,
$$
and the number of such integers $n$ is at most one. 
Next, we assume $a^{2}-4b>0$. Then, 
$$
a_{n}=\frac{(c_{1}-c_{0}\alpha _{-})\alpha _{+}^{n}-(c_{1}-c_{0}\alpha _{+})\alpha _{-}^{n}}{\sqrt{a^{2}-4b}}=0
$$
implies 
\begin{align}
\label{eq:a}
(c_{1}-c_{0}\alpha _{-})\alpha _{+}^{n}=(c_{1}-c_{0}\alpha _{+})\alpha _{-}^{n}.
\end{align}
% By the assumption $ab\not=0$, we have $\alpha _{\pm}\not=0$ which implies that 
By $\alpha _{+}\not=\alpha _{-}$ and (\ref{eq:eq char root nonzero}),
at least one of $c_{1}-c_{0}\alpha _{+}$ and $c_{1}-c_{0}\alpha _{-}$ are not zero. 
The equality (\ref{eq:a}) is equivalent to 
$$
c_{1}-c_{0}\alpha _{-}=(c_{1}-c_{0}\alpha _{+})\left(\frac{\alpha _{-}}{\alpha _{+}}\right)^{n} ,
$$
and the number of such integers $n$ is at most one. 
\end{proof}

\begin{lem}
\label{lem:limit exa}
We assume $c_{0}c_{1}\not=0$. 
% Let $a_{n}$ be the solution of the difference equation (\ref{eq:two order diff eq}). 
The ratio $a_{n+1}/a_{n}$ of consecutive solutions converges if and only if 
$a^{2}-4b\geq 0$ holds. 
Moreover, in this case, the limit is give by 
\begin{align}
\label{eq:limit exa}
\lim_{n \to \infty}
   \frac{a_{n+1}}{a_{n}}
   =
   \begin{cases}
   \alpha & (c_{1}-c_{0}\beta \not=0) \\
   \beta & (c_{1}-c_{0}\beta =0)
   \end{cases}.
\end{align}
\end{lem}
\begin{proof}
First, we prove the ``if'' part. 
We assume $a^{2}-4b<0$, that is, $\alpha _{\pm } \not\in \mathbb{R}$. 
Since there exist $\theta \in \mathbb{R}\setminus \mathbb{Z}$ such that 
$$
\alpha _{\pm }
   =
   |b|^{\frac{1}{2}}e^{\pm \sqrt{-1}\pi \theta } ,
$$
we have 
\begin{align*}
\frac{a_{n+1}}{a_{n}}
   &=
   |b|^{\frac{1}{2}}\frac{c_{0}|b|^{\frac{1}{2}}\sin{(\pi (n+2)\theta )}+(c_{1}-ac_{0})\sin{(\pi (n+1)\theta )}}{c_{0}|b|^{\frac{1}{2}}\sin{(\pi (n+1)\theta )}+(c_{1}-ac_{0})\sin{(\pi n\theta )}}.
\end{align*}
If $\theta =\frac{N}{M}\in \mathbb{Q}$, then $a\not=0$ implies $M\geq 3$, and hence, 
the ratio $a_{n+1}/a_{n}$ does not converge because of $a_{n+M}=(-1)^{N}a_{n}$. 
On the other hand, if we assume $\theta \not\in \mathbb{Q}$, then the ratio $a_{n+1}/a_{n}$ does not converge
by Kronecker's approximation theorem (see \cite[Theorem 439]{HW}). 

Next, we prove the ``only if'' part. 
In the case of $|\alpha |>|\beta |$, we have 
\begin{align*}
\lim_{n \to \infty}
   \frac{a_{n+1}}{a_{n}}
   &=
   \lim_{n \to \infty}
   \frac{(c_{1}-c_{0}\beta )\alpha ^{n+1}-(c_{1}-c_{0}\alpha )\beta ^{n+1}}{(c_{1}-c_{0}\beta )\alpha ^{n}-(c_{1}-c_{0}\alpha )\beta ^{n}} \\
   &=
   \begin{cases}
   \alpha & (c_{1}-c_{0}\beta \not=0) \\
   \beta & (c_{1}-c_{0}\beta =0)
   \end{cases}.
\end{align*}
We assume $|\alpha |=|\beta |$. 
By (\ref{eq:eq cond char root}), we have $\alpha =\beta =\frac{a}{2}$, and hence, we obtain 
\begin{align*}
\lim_{n \to \infty}
   \frac{a_{n+1}}{a_{n}}
   &=
   \lim_{n \to \infty}
   \frac{-c_{0}n\alpha ^{n+1}+c_{1}(n+1)\alpha ^{n}}{c_{0}(1-n)\alpha ^{n}+c_{1}n\alpha ^{n-1}} \\
   &=
   \alpha 
   \lim_{n \to \infty}
   \frac{n(c_{1}-c_{0}\alpha )+c_{1}}{n(c_{1}-c_{0}\alpha )+c_{0}\alpha} \\
   &=
   \begin{cases}
   \alpha & (c_{1}-c_{0}\alpha \not=0) \\
   \frac{c_{1}}{c_{0}}=\alpha & (c_{1}-c_{0}\alpha =0)
   \end{cases}.
\end{align*}
\end{proof}

\section{Proofs of main results}\label{section:proof}
\begin{proof}[Proof of Theorem \ref{thm:monot1-2}]
We prove the theorem by considering the following eight cases: 
\begin{align*}
  &\textrm{(a1)}\ \alpha _{+}>1, \ a^{2}-4b>0, \ a>0;& 
  &\textrm{(a2)}\ \alpha _{+}>1, \ a^{2}-4b>0, \ a<0;&
  &\textrm{(a3)}\ \alpha _{+}>1, \ a^{2}-4b=0;& \\
  &\textrm{(b1)}\ \alpha _{+}<1, \ a^{2}-4b>0, \ a>0;& 
  &\textrm{(b2)}\ \alpha _{+}<1, \ a^{2}-4b>0, \ a<0;&
  &\textrm{(b3)}\ \alpha _{+}<1, \ a^{2}-4b=0;& \\
  &\textrm{(c1)}\ \alpha _{+}=1, \ a^{2}-4b>0;& 
  &\textrm{(c2)}\ \alpha _{+}=1, \ a^{2}-4b=0.&
  && 
\end{align*}
(a1) 
By (\ref{eq:expression 1 final}), we have
\begin{align}
\label{eq:a12}
a_{n+1}-a_{n}
   &=
   \frac{\alpha _{+}^{n}(\alpha _{+}-1)(c_{1}-c_{0}\alpha _{-})-\alpha _{-}^{n}(\alpha _{-}-1)(c_{1}-c_{0}\alpha _{+})}{\sqrt{a^{2}-4b}} % \\
%    &=
%    \frac{\alpha _{+}^{n}}{\sqrt{a^{2}-4b}}
%    \left((\alpha _{+}-1)(c_{1}-c_{0}\alpha _{-})-\left(\frac{\alpha _{-}}{\alpha _{+}}\right)^{n}(\alpha _{-}-1)(c_{1}-c_{0}\alpha _{+})\right).
\\
\label{eq:a1}
   &=
   \frac{\alpha _{+}^{n}}{\sqrt{a^{2}-4b}}
   \left[(\alpha _{+}-1)(c_{1}-c_{0}\alpha _{-})-\left(\frac{\alpha _{-}}{\alpha _{+}}\right)^{n}(\alpha _{-}-1)(c_{1}-c_{0}\alpha _{+})\right]
\end{align}
Since the inequality (\ref{eq:alpha ineq}) $|\alpha _{+}|> |\alpha _{-}|$ holds by the assumption $a>0$,  
the absolute value of the second term of the right-hand side converges to zero. 
% \begin{align}
% \label{eq:a1}
% a_{n+1}-a_{n}
%    &=
%    \frac{\alpha _{+}^{n}}{\sqrt{a^{2}-4b}}
%    \left[(\alpha _{+}-1)(c_{1}-c_{0}\alpha _{-})-\left(\frac{\alpha _{-}}{\alpha _{+}}\right)^{n}(\alpha _{-}-1)(c_{1}-c_{0}\alpha _{+})\right]
% \end{align}
Therefore, $a_{n}\leq a_{n+1}$ holds for sufficiently large $n$ if and only if 
$(\alpha _{+}-1)(c_{1}-c_{0}\alpha _{-})> 0$ holds. 
\\
(a2)
We rewrite (\ref{eq:a12}) as 
\begin{align}
\label{eq:a2}
a_{n+1}-a_{n}
   &=
   \frac{\alpha _{-}^{n}}{\sqrt{a^{2}-4b}}
   \left[\left(\frac{\alpha _{+}}{\alpha _{-}}\right)^{n}(\alpha _{+}-1)(c_{1}-c_{0}\alpha _{-})-(\alpha _{-}-1)(c_{1}-c_{0}\alpha _{+})\right] .
\end{align}
The inequality (\ref{eq:alpha ineq2}) $|\alpha _{+}|< |\alpha _{-}|$ holds by the assumption $a<0$. 
Since $\alpha _{+}>1$ and $a=\alpha _{+}+\alpha _{-}<0$ imply $\alpha _{-}<-1$, 
the value of (\ref{eq:a2}) takes both positive and negative numbers. 
Thus, (\ref{eq:monot1 0}) does not hold. 
\\
(a3)
By (\ref{eq:expression 1-1}), we have 
\begin{align}
a_{n+1}-a_{n}
   &=
   c_{0}\alpha_{+} ^{n}((n+2)\alpha_{+} -(n+1))
   +(c_{1}-ac_{0})\alpha_{+} ^{n-1}((n+1)\alpha_{+} -n) \nonumber \\
   &=
   ((n+1)\alpha_{+} -n)(c_{1}-c_{0}\alpha _{-})\alpha_{+} ^{n-1}
   +c_{0}\alpha_{+} ^{n}(\alpha_{+} -1) \nonumber \\
\label{eq:a3}
   &=
%    \alpha _{+}^{n}\left[\left(n+1-\frac{n}{\alpha _{+}}\right)(c_{1}-c_{0}\alpha _{-})+c_{0}(\alpha _{+}-1)\right].
   \alpha _{+}^{n-1}\left[ (n+1)\left( \alpha _{+}-1+\frac{1}{n+1}\right)(c_{1}-c_{0}\alpha _{-})+c_{0}\alpha _{+}(\alpha _{+}-1)\right].
\end{align}
Since the term $(n+1)\left( \alpha _{+}-1+\frac{1}{n+1}\right)$ can take an arbitrarily large number, 
$a_{n}\leq a_{n+1}$ holds for sufficiently large $n$ if and only if $c_{1}-c_{0}\alpha _{-}> 0$ holds. 
\\
(b1)
In the case of $\alpha _{+}<0$, the value of (\ref{eq:a1}) takes both positive and negative numbers, 
and hence, (\ref{eq:monot1 0}) does not hold. 
In the case of $0<\alpha _{+}<1$, the inequality $a_{n}\leq a_{n+1}$ holds for sufficiently large $n$ if and only if 
$(\alpha _{+}-1)(c_{1}-c_{0}\alpha _{-})\geq 0$ holds. 
\\
(b2)
Similarly to (a2), by using $\alpha _{-}<0$ and (\ref{eq:a2}), we can show that (\ref{eq:monot1 0}) does not hold. 
\\
(b3)
For sufficiently large $n$, 
the value of $(n+1)\left(\alpha _{+}-1+\frac{1}{n+1}\right)$ in (\ref{eq:a3}) takes
a negative number with arbitrarily large absolute value. 
In the case of $\alpha _{+}<0$, the inequality (\ref{eq:monot1 0}) does not hold. 
In the case of $0<\alpha _{+}<1$, the inequality $a_{n}\leq a_{n+1}$ holds for sufficiently large $n$ if and only if 
$c_{1}-c_{0}\alpha _{-}\leq 0$ holds. 
\\
(c1)
Because of
$$
b=\alpha _{-}=a-\alpha _{+}=a-1, 
$$
we have 
\begin{align}
\label{eq:c1}
a_{n+1}-a_{n}
   &=
   (a-1)(a_{n}-a_{n-1})
   =
   (a-1)^{n}(a_{1}-a_{0}).
\end{align}
This implies 
$$
a_{0}\leq a_{1}, \, a\geq 1 \quad \Leftrightarrow \quad a_{n+1}-a_{n}\geq 0 \quad (n \in \mathbb{Z}_{\geq 0}).
$$
(c2)
In this case, we have $\alpha _{+}=\alpha _{-}=1$ and $a=2$. 
Thus, (\ref{eq:c1}) implies  
$$
a_{0}\leq a_{1} \quad \Leftrightarrow \quad a_{n+1}-a_{n}\geq 0 \quad (n \in \mathbb{Z}_{\geq 0}).
$$
\end{proof}

\begin{proof}[Proof of Theorem \ref{thm:monot2-2}]
We first assume $c_{1}\not= c_{0}\beta $. 
By (\ref{eq:limit exa}), the ratio $a_{n+1}/a_{n}$ converges to $\alpha$. 
Since we have 
\begin{align*}
% \label{eq:key step 0}
\left|\alpha -\frac{a_{n+2}}{a_{n+1}}\right|
   =
   \left|\alpha -a+b\frac{a_{n}}{a_{n+1}}\right|
   =
   \left|\frac{-b}{\alpha }+b\frac{a_{n}}{a_{n+1}}\right|
   =
   \left|\frac{b}{\alpha }\frac{a_{n}}{a_{n+1}}\right|\left|\alpha -\frac{a_{n+1}}{a_{n}}\right|, 
\end{align*}
the inequality (\ref{eq:monot2}) is equivalent to 
\begin{align}
\label{eq:another monot2}
   % \left|\beta \frac{\alpha ^{n+1}-\beta ^{n+1}}{\alpha ^{n+2}-\beta ^{n+2}}\right|
\left|\frac{b}{\alpha }\frac{a_{n}}{a_{n+1}}\right|
   =
   \left|\beta \frac{a_{n}}{a_{n+1}}\right|
   \leq 1 .
\end{align}
If $a^{2}-4b>0$ holds, it easily follows from $|\alpha |>|\beta |$ that 
(\ref{eq:another monot2}) holds for sufficiently large $n$. 
We assume $a^{2}-4b=0$. 
By $\alpha =\beta $ and (\ref{eq:expression 1-1}), we have 
\begin{align*}
\left|\beta \frac{a_{n}}{a_{n+1}}\right|
   &=
   \left|\alpha \frac{c_{0}(n+1)\alpha ^{n}+(c_{1}-ac_{0})n\alpha ^{n-1}}{c_{0}(n+2)\alpha ^{n+1}+(c_{1}-ac_{0})(n+1)\alpha ^{n}}\right| \\
   &=
   \left|\frac{c_{0}(n+1)\alpha +(c_{1}-ac_{0})n}{c_{0}(n+2)\alpha +(c_{1}-ac_{0})(n+1)}\right| \\
   &=
   \left|\frac{c_{0}\alpha +(c_{1}-c_{0}\alpha )n}{c_{1}+(c_{1}-c_{0}\alpha )n}\right|.
\end{align*}
By the assumption $c_{1}\not= c_{0}\alpha $, we have 
\begin{align*}
\left|\beta \frac{a_{n}}{a_{n+1}}\right|
   &=
   1-\frac{1}{1+\frac{c_{1}}{c_{1}-\alpha c_{0}}\frac{1}{n}}\frac{1}{n}
   \leq 1 
\end{align*}
for sufficiently large $n$, and hence, the inequality (\ref{eq:monot2}) holds. 

Next, we consider the case of $c_{1}= c_{0}\beta $. 
In this case, we have 
$$
a_{n}=c_{0}\beta ^{n}
$$
for any non-negative integer $n$. 
Indeed, the claim obviously holds for $n=0,1$, and if we assume it holds for $n$ and $n+1$ then we have 
\begin{align*}
a_{n+2}
   =
   aa_{n+1}-ba_{n}
   =
   c_{0}\beta ^{n}((\alpha +\beta )\beta -\alpha \beta )
   =
   c_{0}\beta ^{n+2}.
\end{align*}
Therefore, for any non-negative integer $n$, we have 
\begin{align*}
\left|\alpha -\frac{a_{n+1}}{a_{n}}\right|
   =
   \left|\alpha -\beta \right|\geq 0,
\end{align*}
and hence, the inequality (\ref{eq:monot2}) holds. 
\end{proof}

\begin{proof}[Proof of Theorem \ref{thm:monot1}]
We first assume $\alpha _{+}=1$. 
In this case, the ``if'' part of the proof is obvious. 
We show the ``only if'' part. 
By
$$
a_{k+1}-a_{k}
   =
   (a-1)a_{k}-ba_{k-1}
   =
   (a-1)(a_{k}-a_{k-1})\geq 0 ,
$$
we obtain $a-1\geq 0$. 
Thus, the inequality $0<a_{n}\leq a_{n+1}$ holds for any $n\geq k-1$. 

Next, we assume $\alpha _{+}\not=1$. 
The ``if'' part clearly follows from Theorem \ref{thm:monot1-2}. 
We show the ``only if'' part. 
Namely, we assume the condition (\ref{eq:cond1}). 
Note that the inequality (\ref{eq:alpha ineq}) $|\alpha _{+}|=\alpha _{+}\geq |\alpha _{-}|$ holds by
% $1\not=\alpha _{+}>0$, 
$a>0$. 
The condition $a_{k-1}\leq a_{k}\leq a_{k+1}$ implies 
\begin{align*}
a_{k}-a_{k-1}
   &=
   \frac{\alpha _{+}^{k-1}}{\sqrt{a^{2}-4b}}
   \left[(\alpha _{+}-1)(c_{1}-c_{0}\alpha _{-})-\left(\frac{\alpha _{-}}{\alpha _{+}}\right)^{k-1}(\alpha _{-}-1)(c_{1}-c_{0}\alpha _{+})\right]
   \geq 0
, \\
a_{k+1}-a_{k}
   &=
   \frac{\alpha _{+}^{k}}{\sqrt{a^{2}-4b}}
   \left[(\alpha _{+}-1)(c_{1}-c_{0}\alpha _{-})-\left(\frac{\alpha _{-}}{\alpha _{+}}\right)^{k}(\alpha _{-}-1)(c_{1}-c_{0}\alpha _{+})\right]
   \geq 0 .
\end{align*}
In the following, let $m$ be an arbitrary positive integer. 
\\
(A1) 
In the case of $(\alpha _{-}-1)(c_{1}-c_{0}\alpha _{+})\geq 0$, $\frac{\alpha _{-}}{\alpha _{+}}>0$, 
we have 
\begin{align*}
a_{k+m+1}-a_{k+m}
   &=
   \frac{\alpha _{+}^{k+m}}{\sqrt{a^{2}-4b}}
   \left[(\alpha _{+}-1)(c_{1}-c_{0}\alpha _{-})-\left(\frac{\alpha _{-}}{\alpha _{+}}\right)^{k+m}(\alpha _{-}-1)(c_{1}-c_{0}\alpha _{+})\right] \\
   &\geq
   \frac{\alpha _{+}^{k+m}}{\sqrt{a^{2}-4b}}
   \left[(\alpha _{+}-1)(c_{1}-c_{0}\alpha _{-})-\left(\frac{\alpha _{-}}{\alpha _{+}}\right)^{k}(\alpha _{-}-1)(c_{1}-c_{0}\alpha _{+})\right]
   \geq 0.
\end{align*}
(A2) 
We consider the case of $(\alpha _{-}-1)(c_{1}-c_{0}\alpha _{+})\geq 0$, $\frac{\alpha _{-}}{\alpha _{+}}<0$. 
If $k$ is odd, then $a_{k+2m+1}-a_{k+2m}\geq 0$ is obvious, and we have 
\begin{align*}
a_{k+2m}-a_{k+2m-1}
   &=
   \frac{\alpha _{+}^{k+2m-1}}{\sqrt{a^{2}-4b}}
   \left[(\alpha _{+}-1)(c_{1}-c_{0}\alpha _{-})-\left(\frac{\alpha _{-}}{\alpha _{+}}\right)^{k+2m-1}(\alpha _{-}-1)(c_{1}-c_{0}\alpha _{+})\right] \\
   &=
   \frac{\alpha _{+}^{k+2m-1}}{\sqrt{a^{2}-4b}}
   \left[(\alpha _{+}-1)(c_{1}-c_{0}\alpha _{-})-\left(\frac{\alpha _{-}}{\alpha _{+}}\right)^{k-1}(\alpha _{-}-1)(c_{1}-c_{0}\alpha _{+})\right] \\
   & \quad +
   \frac{\alpha _{+}^{k+2m-1}}{\sqrt{a^{2}-4b}}\left(\frac{\alpha _{-}}{\alpha _{+}}\right)^{k-1}\left[1-\left(\frac{\alpha _{-}}{\alpha _{+}}\right)^{2m}\right](\alpha _{-}-1)(c_{1}-c_{0}\alpha _{+})
   \geq 0.
\end{align*}
If $k$ is even, then $a_{k+2m}-a_{k+2m-1}\geq 0$ is obvious, and we have 
\begin{align*}
a_{k+2m+1}-a_{k+2m}
   &=
   \frac{\alpha _{+}^{k+2m}}{\sqrt{a^{2}-4b}}
   \left[(\alpha _{+}-1)(c_{1}-c_{0}\alpha _{-})-\left(\frac{\alpha _{-}}{\alpha _{+}}\right)^{k+2m}(\alpha _{-}-1)(c_{1}-c_{0}\alpha _{+})\right] \\
   &=
   \frac{\alpha _{+}^{k+2m}}{\sqrt{a^{2}-4b}}
   \left[(\alpha _{+}-1)(c_{1}-c_{0}\alpha _{-})-\left(\frac{\alpha _{-}}{\alpha _{+}}\right)^{k}(\alpha _{-}-1)(c_{1}-c_{0}\alpha _{+})\right] \\
   & \quad +
   \frac{\alpha _{+}^{k+2m}}{\sqrt{a^{2}-4b}}
   \left(\frac{\alpha _{-}}{\alpha _{+}}\right)^{k}\left[1-\left(\frac{\alpha _{-}}{\alpha _{+}}\right)^{2m}\right](\alpha _{-}-1)(c_{1}-c_{0}\alpha _{+})
   \geq 0.
\end{align*}
(B1) 
In the case of $(\alpha _{-}-1)(c_{1}-c_{0}\alpha _{+})\leq 0$, $\frac{\alpha _{-}}{\alpha _{+}}>0$, 
the inequality $a_{k+m+1}-a_{k+m} \geq 0$ holds obviously. 
\\
(B2) 
We consider the case of $(\alpha _{-}-1)(c_{1}-c_{0}\alpha _{+})\leq 0$, $\frac{\alpha _{-}}{\alpha _{+}}<0$. 
If $k$ is odd, then $a_{k+2m}-a_{k+2m-1}\geq 0$ is obvious, and we have 
\begin{align*}
a_{k+2m+1}-a_{k+2m}
   &=
   \frac{\alpha _{+}^{k+2m}}{\sqrt{a^{2}-4b}}
   \left[(\alpha _{+}-1)(c_{1}-c_{0}\alpha _{-})-\left(\frac{\alpha _{-}}{\alpha _{+}}\right)^{k+2m}(\alpha _{-}-1)(c_{1}-c_{0}\alpha _{+})\right] \\
   &=
   \frac{\alpha _{+}^{k+2m}}{\sqrt{a^{2}-4b}}
   \left[(\alpha _{+}-1)(c_{1}-c_{0}\alpha _{-})-\left(\frac{\alpha _{-}}{\alpha _{+}}\right)^{k}(\alpha _{-}-1)(c_{1}-c_{0}\alpha _{+})\right] \\
   & \quad +
   \frac{\alpha _{+}^{k+2m}}{\sqrt{a^{2}-4b}}
   \left(\frac{\alpha _{-}}{\alpha _{+}}\right)^{k}\left[1-\left(\frac{\alpha _{-}}{\alpha _{+}}\right)^{2m}\right](\alpha _{-}-1)(c_{1}-c_{0}\alpha _{+})
   \geq 0.
\end{align*}
If $k$ is even, then $a_{k+2m+1}-a_{k+2m}\geq 0$ is obvious, and we have 
\begin{align*}
a_{k+2m}-a_{k+2m-1}
   &=
   \frac{\alpha _{+}^{k+2m-1}}{\sqrt{a^{2}-4b}}
   \left[(\alpha _{+}-1)(c_{1}-c_{0}\alpha _{-})-\left(\frac{\alpha _{-}}{\alpha _{+}}\right)^{k+2m-1}(\alpha _{-}-1)(c_{1}-c_{0}\alpha _{+})\right] \\
   &=
   \frac{\alpha _{+}^{k+2m-1}}{\sqrt{a^{2}-4b}}
   \left[(\alpha _{+}-1)(c_{1}-c_{0}\alpha _{-})-\left(\frac{\alpha _{-}}{\alpha _{+}}\right)^{k-1}(\alpha _{-}-1)(c_{1}-c_{0}\alpha _{+})\right] \\
   & \quad +
   \frac{\alpha _{+}^{k+2m-1}}{\sqrt{a^{2}-4b}}
   \left(\frac{\alpha _{-}}{\alpha _{+}}\right)^{k-1}\left[1-\left(\frac{\alpha _{-}}{\alpha _{+}}\right)^{2m}\right](\alpha _{-}-1)(c_{1}-c_{0}\alpha _{+})
   \geq 0.
\end{align*}
\end{proof}
By the proof of Theorem \ref{thm:monot1}, we conclude that 
if the inequality (\ref{eq:monot1-k}) holds for $n=k-1$, $n=k$ and for sufficiently large $n$, 
then it holds for any $n \geq k-1$. 
\begin{rmk}
\label{rmk:Hartman-Aurel}
In Hartman-Aurel \cite{HA1}, \cite{HA2}, 
some sufficient conditions for the monotone property $y_{n}<y_{n+1}$ of solutions to 
the linear ordinary difference equations are given. 
If we apply the result for our difference equation (\ref{eq:two order diff eq}), 
we obtain a sufficient condition for the property: 
$$
a-1-b> 0, \quad b> 0 .
$$
Indeed, it is obvious since we have 
$$
a_{n+2}-a_{n+1}=(a-1-b)a_{n+1}+b(a_{n+1}-a_{n}). 
$$
\end{rmk}

\begin{proof}[Proof of Corollary \ref{cor:monot1}]
Note that $a_{1}=aa_{0}$. 
We apply Theorem \ref{thm:monot1} for $k=0$. 
The condition (\ref{eq:cond1}) is reduced into 
\begin{align*}
& 0<a_{-1}=0< a_{0}\leq a_{1}=aa_{0}, \quad 1\not=\alpha _{+}>0,\quad a>0, \quad (\alpha _{+}-1)(c_{1}-c_{0}\alpha _{-})=a_{0}\alpha _{+}(\alpha _{+}-1)> 0,
\end{align*}
which is equivalent to $a_{0}>0$, $a\geq 1$, $\alpha _{+}>1$. 
The condition (\ref{eq:cond1 another}) is equivalent to $a_{0}>0$, $a\geq 1$, $\alpha _{+}=1$. 
Summarizing them, we obtain the condition (\ref{eq:cond1 cor}). 
\end{proof}
\begin{rmk}
\label{rmk:domain D1}
Note that under the condition (\ref{eq:cond1 another}) in Corollary \ref{cor:monot1}, we have 
$$
\alpha _{+}=\alpha , \quad \alpha _{-}=\beta .
$$
In this situation, two domains 
\begin{align}
\label{eq:domain D1}
D_{1}&:=\{(\alpha ,\beta ) \in \mathbb{R}^{2} \mid \alpha +\beta \geq 1, \, \alpha \geq 1, \, |\alpha |\geq |\beta |\}
,\\
\label{eq:domain D1prime}
D_{1}^{\prime }&:=\{(a,b) \in \mathbb{R}^{2} \mid a \geq 1, \, \alpha \geq 1, \, |\alpha |\geq |\beta |\}
\end{align}
are drawn as in Figures \ref{fig:D1} and \ref{fig:Dab1}, respectively. 

\begin{figure}[h]
  \begin{minipage}[b]{0.45\linewidth}
    \centering
    \includegraphics[keepaspectratio, scale=0.8]{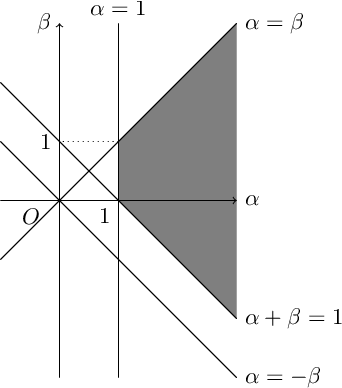}
    \caption{$D_{1}$}
    \label{fig:D1}
  \end{minipage}
  \begin{minipage}[b]{0.45\linewidth}
    \centering
    \includegraphics[keepaspectratio, scale=0.8]{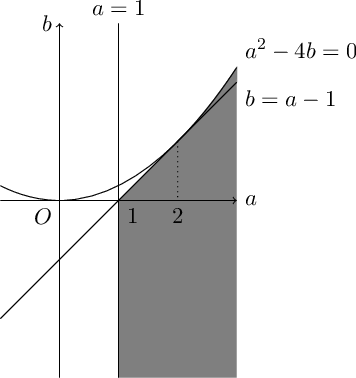}
    \caption{$D_{1}^{\prime }$}
    \label{fig:Dab1}
  \end{minipage}
\end{figure}
\end{rmk}

\begin{proof}[Proof of Theorem \ref{thm:monot2}]
First, we show the ``only if'' part of the proof. 
Namely, we assume $a^{2}-4b \geq 0$ and $|\alpha +\beta |=|a|\geq |\beta |$. 
Similarly to the proof of Theorem \ref{thm:monot2-2}, for any $n$, 
the inequality (\ref{eq:monot2}) is equivalent to (\ref{eq:another monot2}), that is, 
$$
\left|\frac{b}{\alpha }\frac{a_{n}}{a_{n+1}}\right|
   =
   \left|\beta \frac{\alpha ^{n+1}-\beta ^{n+1}}{\alpha ^{n+2}-\beta ^{n+2}}\right|\leq 1 .
$$
We consider the following five cases: 
\begin{align*}
&\textrm{(i)}\ |\alpha |=|\beta |; \qquad 
\textrm{(ii-1)}\ |\alpha |>|\beta |,\ \alpha >0,\ \beta > 0 ;\qquad 
\textrm{(ii-2)}\ |\alpha |>|\beta |,\ \alpha <0,\ \beta < 0 ;\\
&\textrm{(ii-3)}\ |\alpha |>|\beta |,\ \alpha >0,\ \beta < 0 ;\qquad 
\textrm{(ii-4)}\ |\alpha |>|\beta |,\ \alpha <0,\ \beta > 0 .
\end{align*}
(i)
Note that $\alpha =\beta \not=0$ by (\ref{eq:eq cond char root}). 
Since we have 
$$
a_{n}=a_{0}(n+1)\alpha ^{n}, \quad \frac{a_{n+1}}{a_{n}}=\frac{n+2}{n+1}\alpha , 
$$
we obtain 
\begin{align*}
\left|\alpha -\frac{a_{n+2}}{a_{n+1}}\right|
   =
   \frac{|\alpha |}{n+2}
   \leq 
   \frac{|\alpha |}{n+1}
   =
   \left|\alpha -\frac{a_{n+1}}{a_{n}}\right|
\end{align*}
for any non-negative integer $n$. 
\\
(ii-1)
In this case, we have $\alpha > \beta  >0$. 
The condition (\ref{eq:cond2}) is reduced into $\alpha +\beta \geq \beta $, and hence, we have 
\begin{align*}
|\alpha ^{n+2}-\beta ^{n+2}|-|\beta (\alpha ^{n+1}-\beta ^{n+1})|
   =
   \alpha ^{n+1}(\alpha -\beta )\geq 0.
\end{align*}
(ii-2)
The inequality 
(\ref{eq:another monot2})
% (\ref{eq:monot2}) 
can be shown similarly to (ii-1). 
\\
(ii-3)
We have $|\alpha |=\alpha >-\beta =|\beta |$ and $|\alpha +\beta |=\alpha +\beta \geq -\beta =|\beta |$. 
Thus, we obtain 
\begin{align}
|\alpha ^{n+2}-\beta ^{n+2}|-|\beta (\alpha ^{n+1}-\beta ^{n+1})|
   &=
   \alpha ^{n+1}(\alpha +\beta )-2\beta ^{n+2} \nonumber \\
\label{eq:key step}
   &=
   \alpha ^{n+1}(\alpha +2\beta )-\beta (\alpha ^{n+1}+2\beta ^{n+1}).
\end{align}
Since the inequality $\alpha ^{n+1}+2\beta ^{n+1}\geq 0$ holds for any non-negative integer $n$, 
the right-hand side of (\ref{eq:key step}) is non-negative. 
\\
(ii-4)
By setting $\alpha ^{\prime}:=-\alpha $ and $\beta ^{\prime}:=-\beta $, 
we can prove (\ref{eq:another monot2}) similarly to (ii-3). 

Next, we show the ``if'' part. 
The inequality (\ref{eq:another monot2}) for $n=0$ is reduced into 
$$
\left|\frac{\beta }{\alpha +\beta }\right|\leq 1 ,
$$
which implies $\alpha +\beta \not=0$ and $|\beta |\leq |\alpha +\beta |$. 
Further, by Theorem \ref{thm:monot2-2}, the inequality (\ref{eq:monot2}) holds for 
sufficiently large $n$ if and only if $a^{2}-4b\geq 0$ holds. 
\end{proof}

\begin{rmk}
As an analogy of Theorem \ref{thm:monot2}, we consider 
a necessary condition for arbitrary initial values: 
$a^{2}-4b\geq 0$ and 
$$
\left|\frac{\beta c_{0}}{c_{1}}\right|\leq 1 .
$$
However, it is not a sufficient condition. 
Indeed, if we set 
$$
c_{0}=1, \quad c_{1}=3, \quad \alpha =2.1, \quad \beta =-2 ,
$$
then we have $a_{2}=(2.1-2)3+4.2=4.5$ and 
$$
\left|\frac{\beta a_{0}}{a_{1}}\right|=\left|\frac{\beta c_{0}}{c_{1}}\right|=\frac{2}{3}\leq 1 ,\qquad 
\left|\frac{\beta a_{1}}{a_{2}}\right|=\frac{4}{3}\geq 1 .
$$
\end{rmk}

\begin{rmk}
\label{rmk:domain D2}
Two domains 
\begin{align}
\label{eq:domain D2}
D_{2}&:=\{(\alpha ,\beta ) \in \mathbb{R}^{2} \mid |\alpha +\beta |\geq |\beta |, |\alpha |\geq |\beta |\},
\\
\label{eq:domain D2prime}
D_{2}^{\prime }&:=\{(a,b) \in \mathbb{R}^{2} \mid |a|\geq |\beta |, |\alpha |\geq |\beta |\}
\end{align}
corresponding to the condition (\ref{eq:cond2}) of Theorem \ref{thm:monot2} 
are drawn as in Figures \ref{fig:D2} and \ref{fig:Dab2}, respectively. 

\begin{figure}[h]
  \begin{minipage}[b]{0.45\linewidth}
    \centering
    \includegraphics[keepaspectratio, scale=0.8]{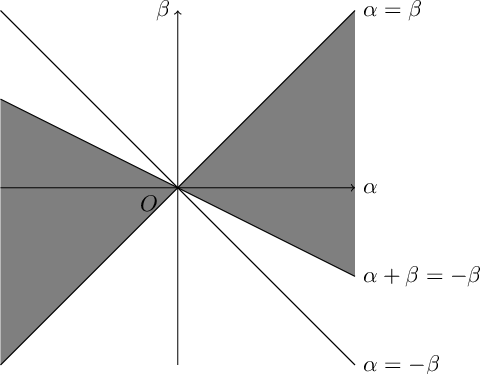}
    \caption{$D_{2}$}
    \label{fig:D2}
  \end{minipage}
  \begin{minipage}[b]{0.45\linewidth}
    \centering
    \includegraphics[keepaspectratio, scale=0.8]{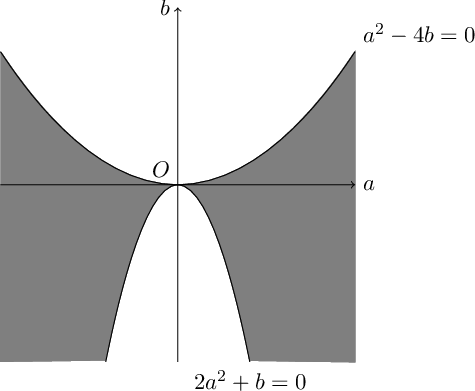}
    \caption{$D_{2}^{\prime }$}
    \label{fig:Dab2}
  \end{minipage}
\end{figure}
\end{rmk}

\begin{rmk}
\label{rmk:Ricacci}
Theorem \ref{thm:monot2} can be translated into the property of 
a solution to a difference Riccati equation with real constant coefficients: 
\begin{align}
\label{eq:diff Ricacci1}
b_{n+1}=\frac{xb_{n}+y}{zb_{n}+w}.
\end{align}
Indeed, by setting 
$$
b_{n}:=\frac{a_{n+1}}{a_{n}} 
$$
and dividing the difference equation (\ref{eq:two order diff eq}) by $a_{n+1}$, we obtain 
\begin{align}
\label{eq:diff Ricacci2}
b_{n+1}=a-b\frac{1}{b_{n}}=\frac{ab_{n}-b}{b_{n}} .
\end{align}
Theorem \ref{thm:monot2} gives a condition under which 
the solution of the difference Riccati equation (\ref{eq:diff Ricacci2}) 
with initial value
$$
b_{0}=\frac{a_{1}}{a_{0}}=a
$$
approaches to the fixed point monotonically.  
\end{rmk}

\begin{proof}[Proof of Theorem \ref{thm:monot3}]
For the solution $a_n$ to the difference equation (\ref{eq:two order diff eq}) with 
arbitrary initial values, we have 
\begin{align*}
\left|a_{n}\alpha -a_{n+1}\right|
   &=
   \left|c_{0}(h_{n}(\alpha ,\beta )\alpha -h_{n+1}(\alpha ,\beta ))+(c_{1}-ac_{0})(h_{n-1}(\alpha ,\beta )\alpha -h_{n}(\alpha ,\beta ))\right| \\
   &=
   \left|c_{0}\frac{\beta ^{n+2}-\alpha \beta ^{n+1}}{\alpha -\beta }
   +(c_{1}-ac_{0})\frac{\beta ^{n+1}-\alpha \beta ^{n}}{\alpha -\beta }\right| \\
   &=
   \left|-c_{0}\beta ^{n+1}
   -(c_{1}-ac_{0})\beta ^{n}\right| \\
   &=
   |\beta |^{n}
   \left|c_{0}\beta 
   +(c_{1}-ac_{0})\right| \\
   &=
   |c_{1}-c_{0}\alpha |
   |\beta |^{n}.
\end{align*}
If we assume 
\begin{align}
\left|a_{n}\alpha -a_{n+1}\right|-\left|a_{n+1}\alpha -a_{n+2}\right|
   &=
   |c_{1}-c_{0}\alpha |
   |\beta |^{n}
   -|c_{1}-c_{0}\alpha |
   |\beta |^{n+1} \nonumber \\
\label{eq:monotonic3 suf}
   &=
   |c_{1}-c_{0}\alpha |
   |\beta |^{n}   
   (1-|\beta |)\geq 0 ,
\end{align}
then we obtain $|\beta |\leq 1$. 
Conversely, if we assume $|\beta |\leq 1$, then 
the inequality (\ref{eq:monotonic3 suf}) holds. 
\end{proof}

\begin{rmk}
If we assume that the monotone property 1 (\ref{eq:monot1}) 
and the monotone property 3 (\ref{eq:monot3}) hold for any non-negative integer $n$, 
then the monotone property 2 (\ref{eq:monot2}) also holds, since we have
$$
\left|\alpha -\frac{a_{n+2}}{a_{n+1}}\right|\leq \left|\frac{a_{n}}{a_{n+1}}\right|\left|\alpha -\frac{a_{n+1}}{a_{n}}\right|\leq \left|\alpha -\frac{a_{n+1}}{a_{n}}\right|
$$
(see Proposition \ref{prop:domains}). 
However, the converse is not true. 
For example, 
although the solution in Example \ref{exa:count exam1} satisfies (\ref{eq:monot2}) and (\ref{eq:monot3}), 
it does not satisfy (\ref{eq:monot1}). 
\end{rmk}

\begin{rmk}
\label{rmk:Pisot}
We assume $a,b \in \mathbb{Z}$. 
Suppose that the characteristic polynomial $x^{2}-ax+b$ is irreducible over $\mathbb{Z}$. 
In this case, the condition (\ref{eq:cond3}) is $\beta \not=0$ and $|\beta |<1$, 
which implies that $\alpha $ is a quadratic integer greater than $1$ and is a Pisot number \cite{BDGPS}. 
Namely, 
the inequality (\ref{eq:monot3}) holds for any initial values $c_{0},c_{1}$ and for any $n$ 
if and only if $\alpha$ is a Pisot number. 
\end{rmk}

\begin{rmk}
\label{rmk:domain D3}
Two domains 
\begin{align}
\label{eq:domain D3}
D_{3}&:=\{(\alpha ,\beta ) \in \mathbb{R}^{2} \mid |\beta |\leq 1, |\alpha |\geq |\beta |\}
,\\
\label{eq:domain D3prime}
D_{3}^{\prime }&:=\{(a,b) \in \mathbb{R}^{2} \mid |\beta |\leq 1, |\alpha |\geq |\beta |\}
\end{align}
corresponding the condition (\ref{eq:cond3}) of Theorem \ref{thm:monot3}
are drawn as in Figures \ref{fig:D3} and \ref{fig:Dab3}, respectively. 

\begin{figure}[h]
  \begin{minipage}[b]{0.45\linewidth}
    \centering
    \includegraphics[keepaspectratio, scale=0.8]{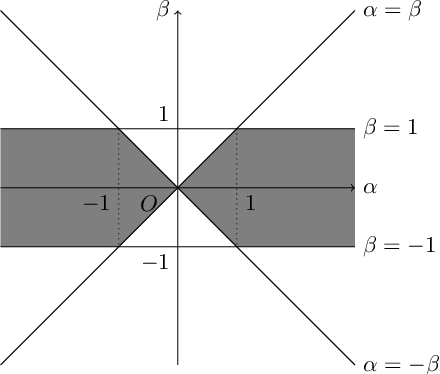}
    \caption{$D_{3}$}
    \label{fig:D3}
  \end{minipage}
  \begin{minipage}[b]{0.45\linewidth}
    \centering
    \includegraphics[keepaspectratio, scale=0.8]{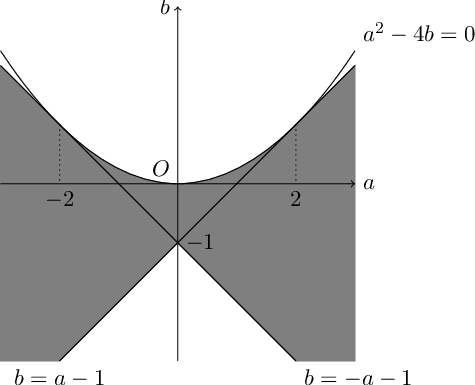}
    \caption{$D_{3}^{\prime }$}
    \label{fig:Dab3}
  \end{minipage}
\end{figure}
\end{rmk}

\section{Concluding remarks}\label{section:concluding}
One of the motivations of our study is 
to characterize the Fibonacci numbers by monotone properties. 
As a generalization of the Fibonacci numbers, 
the solution $a_n$ to the difference equation (\ref{eq:two order diff eq}), 
especially that with initial values $a_{-1}:=0$, $a_{0}:=1$ (i.e., 
the special values of the complete homogeneous symmetric polynomials), 
has been studied well (e.g., \cite{K}). 
On the other hand, in previous studies, it seems that 
the monotone properties 1--3  
% the monotonicities (\ref{eq:monotonic 1}), (\ref{eq:monotonic 2}) and (\ref{eq:monotonic 3}) 
were not paid attention to. 
By our main results, we determine integer solutions $a_n$ satisfying 
the monotone properties (\ref{eq:monot1}), (\ref{eq:monot2-2}) and (\ref{eq:monot3}) which are analogies of 
the inequalities (\ref{eq:monotonic 1}), (\ref{eq:monotonic 2}) and (\ref{eq:monotonic 3}) for the Fibonacci numbers.

Indeed, 
it suffices to restrict the characteristic roots (resp. the coefficients) of (\ref{eq:two order diff eq})
to the intersection 
$D:=D_{1}\cap D_{2}\cap D_{3}$ (resp. $D^{\prime}:=D_{1}^{\prime }\cap D_{2}^{\prime }\cap D_{3}^{\prime }$) of 
the domains $D_{1}$, $D_{2}$ and $D_{3}$ (resp. $D_{1}^{\prime }$, $D_{2}^{\prime }$ and $D_{3}^{\prime }$)
in Remarks \ref{rmk:domain D1}, \ref{rmk:domain D2} and \ref{rmk:domain D3}. 
By easy calculations, we obtain the following proposition. 
\begin{prop}
\label{prop:domains}
We have
\begin{align}
D&:=D_{1}\cap D_{2}\cap D_{3}
   =
   \{(\alpha ,\beta ) \in \mathbb{R}^{2} \mid \alpha +\beta  \geq 1, \, \alpha \geq 1, \, |\beta |\leq 1\}, \\
D^{\prime}&:=D_{1}^{\prime }\cap D_{2}^{\prime }\cap D_{3}^{\prime }
   =
%    \{(a ,b ) \in \mathbb{R}^{2} \mid a \geq 1, \, b\leq a-1, \, b\geq -a-1\}.
   \{(a ,b ) \in \mathbb{R}^{2} \mid a \geq 1, \, a-1 \geq b\geq -a-1\}.
\end{align}
\begin{figure}[h]
  \begin{minipage}[b]{0.45\linewidth}
    \centering
    \includegraphics[keepaspectratio, scale=0.8]{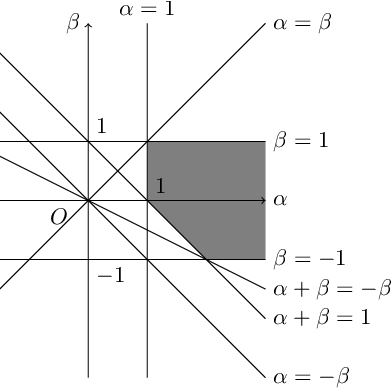}
    \caption{$D$}
  \end{minipage}
  \begin{minipage}[b]{0.45\linewidth}
    \centering
    \includegraphics[keepaspectratio, scale=0.8]{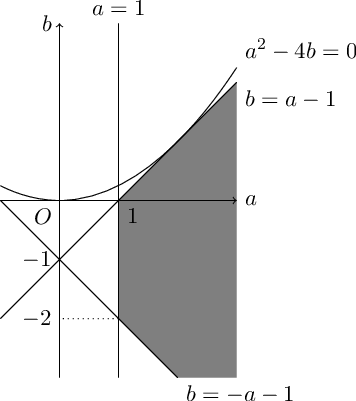}
    \caption{$D^{\prime}$}
  \end{minipage}
\end{figure}
\end{prop}

Further, we note the following lemma. 
\begin{lem}
Let $(a,b) \in \mathbb{Z}^{2}\cap D^{\prime}$. 
Then, we have the following equivalence: 
\begin{align}
\label{eq:irr cond}
x^{2}-ax+b\text{ is irreducible over }\mathbb{Z} \quad \Leftrightarrow  \quad b\not=-a-1, \, 0, \, a-1.
\end{align}
\end{lem}
\begin{proof}
It is obvious that if $b=-a-1, \, 0, \, a-1$, 
the polynomial $x^{2}-ax+b$ is reducible over $\mathbb{Z}$. 
We prove the converse. Assume that $x^{2}-ax+b$ is reducible over $\mathbb{Z}$. 
Then, there exists an integer $k$ such that 
$$
x^{2}-ax+b=(x-k)(x-a+k).
$$
By 
$(k,a-k) \in \mathbb{Z}^{2}\cap D'$ or $(a-k,k) \in \mathbb{Z}^{2}\cap D'$, 
Proposition \ref{prop:domains} implies that $|a-k|\leq 1$ or $|k|\leq 1$ holds. 
Therefore, we obtain $k=a-1, a, a+1$ or $k=-1,0,1$, and hence, 
the value of $b=k(a-k)$ takes only $b=-a-1,0,a-1$. 
\end{proof}

We list difference equations (\ref{eq:two order diff eq}) such that 
the coefficients $(a,b) \in \mathbb{Z}^{2}\cap D^{\prime}$ satisfy the condition (\ref{eq:irr cond}): 
\begin{align*}
a=1\,&:\,a_{n+2}=a_{n+1}+a_{n}, \\
a=2\,&:\,a_{n+2}=2a_{n+1}+2a_{n}, \,a_{n+2}=2a_{n+1}+a_{n},\\
a=3\,&:\,a_{n+2}=3a_{n+1}+3a_{n}, \,a_{n+2}=3a_{n+1}+2a_{n}, \,a_{n+2}=3a_{n+1}+a_{n}, \,a_{n+2}=3a_{n+1}-a_{n},\\
\vdots &
\end{align*}
The first equation is the recursion of the Fibonacci numbers. 
The second equation in $a=2$ is the recursion of the Pell numbers.
The recursion of the Fibonacci numbers is characterized by 
\begin{align*}
  \{ (a,b) \in \mathbb{Z}^2 \cap \partial D'  \mid \text{$(a,b)$ satisfies (\ref{eq:irr cond})} \}
  =\{ (1,-1) \},
\end{align*}
where $\partial D'$ means the boundary of the domain $D'$. 
From the viewpoint of the monotone properties and irreducibility over $\mathbb{Z}$,
the Fibonacci numbers are very special.
Thus, the above recursions with initial values $a_{-1}:=0$, $a_{0}:=1$
give appropriate generalizations of the Fibonacci numbers.

\begin{Acknowledgement}
The authors are grateful to Professor Shigeki Akiyama who gave an opportunity 
to consider Theorem \ref{thm:monot3}. 
We thank Professor Kentaro Mitsui for pointing out the interpretation in Remark \ref{rmk:Ricacci}. 
This work was supported by JST CREST Grant Number JP19209317 and JSPS KAKENHI Grant Number 21K13808. 
\end{Acknowledgement}

\end{document}